\documentclass[8pt]{article}
\usepackage{amsmath}
\usepackage{amsthm}
\usepackage{amssymb,amsCD}
\usepackage{amsthm}
\usepackage{newlfont}
\usepackage[bookmarksnumbered,  plainpages]{hyperref}

\newtheorem{thm}{Theorem}[section]
 \newtheorem{ex}{Example}[section]
 \newtheorem{cor}[thm]{Corollary}
 \newtheorem{lem}[thm]{Lemma}
 \newtheorem{prop}[thm]{Proposition}
 \theoremstyle{definition}
 \newtheorem{defn}{Definition}[section]
 \theoremstyle{remark}
 \newtheorem{rem}[thm]{Remark}
\begin{document}
	
	\title{Sufficient conditions for the unique solvability of absolute value matrix equations}
	
	%\author{Shubham Kumar$^{a,1}$, Deepmala$^{a,2}$\\
	%	\emph{\small $^{a}$Mathematics Discipline,}
	%	\emph{\small PDPM-IIITDM Jabalpur, M.P. India}\\
	%\emph{\small Jabalpur - 482005 (MP), India}\\
	%\emph{\small $^{b}$}\\
	%	\emph{\small $^1$Email: shub.srma@gmail.com, $^2$Email: dmrai23@gmail.com, }}
	
	\author{Shubham Kumar$^{1,\ast}$, Deepmala$^{2, \star}$ \\
		\emph{\small $^{1,2}$Mathematics Discipline,}\\	
		\emph{\small  PDPM-Indian Institute of Information Technology, Design and Manufacturing, Jabalpur, M.P., India}\\
		\emph{\small $^{\ast}$ Email: shub.srma@gmail.com, ORCID: https://orcid.org/0000-0001-5237-9719, }\\
		\emph{\small $^{\star}$Email: dmrai23@gmail.com, ORCID: https://orcid.org/0000-0002-2600-6836.}}
	\date{}
	\maketitle

% % % % % % % % % % % % % % % % % % % % % % % % % % % % % % % % % % % % % % % % % % % % %
% % % % % % % % % % % % % % % % % % % % % % % % % % % % % % % % % % % % % % % % % % % % %
% % % % % % % % % % % % % % % % % % % % % % % % % % % % % % % % % % % % % % % % % % % % %

\section{Abstract} 
	In this paper, we discussed the unique solvability of the two absolute value matrix equation. The unique solvability condition $\rho (\vert A^{-1} B \vert)<1$  is provided for the generalized absolute value matrix equation (GAVME) $AX + B \vert X \vert = F$. This condition is superior to that of Kumar et al. [J. Numer. Anal. Approx. Theory, 51(1) (2022) 83-87]. We also discussed different conditions for the unique solvability of the new generalized absolute value matrix equation (NGAVME) $AX+B\vert CX \vert=F$ with $A,B,C,F,X \in \mathcal{R}^{n \times n}$. We also provided the corrected version of Corollary 2.1 from the published work by Wang et al. [Appl. Math. Lett., 116 (2021) 106966].

\noindent \textbf{Keywords.}  Absolute value matrix equations. Unique solution. Sufficient condition.

\noindent \textbf{2020 MSC.}  15A06, 15A18, 90C30.
%65F10 --- Iterative numerical methods for linear systems.
%90C05 --- Linear programming.
%90C30 --- Nonlinear programming.
%15A18 --- Eigenvalues, singular values, and eigenvectors.
%15A06 --- Linear Equations.

% % % % % % % % % % % % % % % % % % % % % % % % % % % % % % % % % % % % % % % %

% % % % % % % % % % % % % % % % % % % % % % % % % % % % % % % % % % % % % % % %

% % % % % % % % % % % % % % % % % % % % % % % % % % % % % % % % % % % % % % % % % % % % %
% % % % % % % % % % % % % % % % % % % % % % % % % % % % % % % % % % % % % % % % % % % % %
% % % % % % % % % % % % % % % % % % % % % % % % % % % % % % % % % % % % % % % % % % % % %

% % % % % % % % % % % % % % % % % % % % % % % % % % % % % % % % % % % % % % % %

%\footnotetext[1]{Corresponding author.}

% % % % % % % % % % % % % % % % % % % % % % % % % % % % % % % % % % % % % % % %
\section{Introduction}
The NGAVME is defined as
\begin{equation}\label{Equ NGAVME}
	AX+B\vert CX \vert=F,
\end{equation}
%%%%%%%%%%%%%%%%%%%%%%%%%%%%%%%%%%%%%%%%%%%%%%%%%%%%%%%%%%%%%%%%%%%%%%%%%%%%%%%%%%%
 which is a generalization form of the following GAVME
	  \begin{equation} \label{Equ GAVME}
	AX + B\vert X \vert =F,
\end{equation}
with $A,B,C,F \in \mathcal{R}^{n\times n}$ are given and $X \in \mathcal{R}^{n\times n}$ is to be determined. For a matrix  $A \in \mathcal{R}^{n\times n}$ and a vector $x \in \mathcal{R}^{n}$, $\vert  A \vert$ and $\vert  x \vert$ denotes the component-wise absolute value of the matrix and the vector, respectively.

%%%%%%%%%%%%%%%%%%%%%%%%%%%%%%%%%%%%%%%%%%%%%%%%%%%%%%%%%%%%%%%%%%%%%%%%%%%%%%%%%%%

GAVME and NGAVME are the generalizations of the following generalized absolute value equation (GAVE)
	  \begin{equation} \label{Equ GAVE}
	Ax + B\vert x \vert =f,
\end{equation}
where $A,B \in \mathcal{R}^{n\times n}$, $f \in \mathcal{R}^{n}$ are given and $x \in \mathcal{R}^{n}$ is unknown.

The Sylvester-like absolute value equation is defined as
\begin{equation}\label{Equ Sylvester-like AVE}
	AXK+B\vert X \vert L=F,
\end{equation}
with the matrices $A,B,K,L,F \in \mathcal{R}^{n \times n}$ are known and $X \in \mathcal{R}^{n \times n}$ to be determined.

%%%%%%%%%%%%%%%%%%%%%%%%%%%%%%%%%%%%%%%%%%%%%%%%%%%%%%%%%%%%%%%%%%%%%%%%%%%%%%%%%%%%%%%%

Rohn \cite{rohn2004theorem} was the first to consider the GAVE (\ref{Equ GAVE}), then after Mangasarian et al. \cite{mangasarian2006absolute,mangasarian2007programming} studied about GAVE in detail. The absolute value equations are an active area of research because of their applications in different branches of mathematics, for instance, optimization problems like linear and quadratic programs, linear interval equations, and finding Nash-equilibrium points in the bi-matrix game. For more details and applications about different types of absolute value equations, readers may refer to (\cite{mezzadri2020solution,wu2021unique,wu2021uniquengave} and the references therein).

In the literature, Baksalary et al. \cite{baksalary1979matrix,baksalary1980matrix} consider the matrix equations $AX-BY=C$ and $AXB+CYD=E$ respectively, and established a necessary and sufficient condition for its solvability. Roth \cite{roth1952equations}, give the solvability conditions for the matrix equations $AX-YB=C$ and $AX-XB=C$. For the equation $AXB+CYD=E$, Peng et al. \cite{peng2006efficient} provided an iterative method for their solution. Dehghan and Hajarian \cite{dehghan2008iterative,dehghan2010general} discussed the numerical method for the general coupled matrix equations. The unique solvability conditions of the Sylvester-like AVE (\ref{Equ Sylvester-like AVE}) are discussed in (\cite{hashemi2021sufficient,wang2021new} and references therein). Li \cite{li2022sufficient} considered the new class of Sylvester-like AVE $AXK+\vert BXL \vert=F$, and sufficient conditions provided for their unique solvability. Dehghan et al. \cite{dehghan2020matrix} most recently expanded upon the idea of absolute value equations for matrix equations.

  Influenced by the work in \cite{zhou2021unique}, with Sylvester-like AVE, we further generalized the concept of GAVME and considered the NGAVME. These Matrix equations are observed in many branches of engineering, for example, in the field of interval matrix equations \cite{neumaier1990interval,seif1994interval}, robust control\cite{shashikhin2002robust} and so on. For the many forms of the matrix equations, several results of solvability and numerical methods found in (\cite{baksalary1979matrix,baksalary1980matrix,dehghan2008iterative,dehghan2010general,dehghan2020matrix,hashemi2021sufficient,peng2006efficient,roth1952equations,tang2022solvability,wang2021new,xie2021unique} and the references therein).

The appearance of the non-differentiable and non-linear terms in NGAVME (\ref{Equ NGAVME}), GAVME (\ref{Equ GAVME}) and Sylvester-like AVE (\ref{Equ Sylvester-like AVE}) make the study challenging and interesting. To the best of our knowledge, nobody has studied the NGAVME (\ref{Equ NGAVME}) yet. Based on this, we present sufficient conditions for unique solvability for the NGAVEM and the GAVME in this paper.

The structure of this work is as follows: in Sect.2, we recall some results and definitions. In Sect.3, we present some results for the unique solvability of the NGAVME (\ref{Equ NGAVME}) and the unique solution condition obtains for the GAVME (\ref{Equ GAVME}). Finally, we conclude our discussion in Sect.4.

%%%%%%%%%%%%%%%%%%%%%%%%%%%%%%%%%%%%%%%%%%%%%%%%%%%%%%%%%%%%%%%%%%%%%%%%%%%%%%%%%%%%%%%%
\section{Preliminaries} \label{section 2}
For later usage, we review several definitions and findings in this section. We use the vec operator and Kronecker product $`\otimes$' of the matrices; for the definition and properties of the vec operator and Kronecker product, one can see in \cite{hashemi2021sufficient,horn2012matrix,tang2022solvability}.

\begin{defn} \cite{sznajder1995generalizations}
	Let $\mathcal{M} =\{M_1,M_2\}$ denote the set of matrices with $M_1,M_2$ $\in$ $\mathcal{R}^{n \times n}$. A matrix R $\in$ $\mathcal{R}^{n \times n}$ is called a row representative (or column representative) of $\mathcal{M}$, if $R_{j.}$  $\in$ $\{(M_{1})_{j.},(M_{2})_{j.}\}$ (or $R_{.j}$  $\in$ $\{(M_{1})_{.j},(M_{2})_{.j}\}$) j=1,2,...,n, where $R_{j.}$, $(M_{1})_{j.}$, and $(M_{2})_{j.}$ (or $R_{.j}$, $(M_{1})_{.j}$, and $(M_{2})_{.j}$) denote the $j^{th}$ row (or column) of R, $M_{1}$ and $M_{2}$, respectively.
\end{defn} 

% % % % % % % % % % % % % % % % % % % % % % % % % % % % % % % % % % % % % % % % % % % % %
\begin{defn} \cite{sznajder1995generalizations}	
If the determinants of each row (or column) representative matrices of $\mathcal{M} = \{M_1,M_2\}$ are positive then $\mathcal{M}$ is said to be hold the row (or column) $\mathcal{W}$-property.

%	The set $\mathcal{M} = \{M_1,M_2\}$ is said to be hold the row (or column) $\mathcal{W}$-property if the determinants of all row (or column) representative matrices of $\mathcal{M}$ are positive.	
\end{defn}

% % % % % % % % % % % % % % % % % % % % % % % % % % % % % % % % % % % % % % % % % % % % %

\begin{thm}\label{thm 01} \cite{horn2012matrix}
	For non-negative matrices $A, B \in \mathcal{R}^{n \times n},$ if $0 \leq A \leq B$, then $\rho(A) \leq \rho(B),$ where $\rho(.)$ denote the spectral radius of the matrix. 
\end{thm}

% % % % % % % % % % % % % % % % % % % % % % % % % % % % % % % % % % % % % % % % % % % % %

\begin{thm}\label{thm 001} \cite{horn2012matrix}
For a matrix $A \in \mathcal{R}^{n \times n}$, there is a matrix norm $\vert \vert \vert . \vert \vert \vert $ such that $\rho(A)$ $\leq$ $ \vert \vert \vert A \vert \vert \vert $.
\end{thm}

% % % % % % % % % % % % % % % % % % % % % % % % % % % % % % % % % % % % % % % % % % % % %
\begin{thm}\label{thm 1} \cite{dehghan2020matrix,kumar2022note,kumar2023note,xie2021unique}
	If any of the following requirements are fulfilled, the GAVME (\ref{Equ GAVME}) has exactly one solution:\\
	(i)  $\sigma_{max}(\vert B \vert) < \sigma_{min}(A);$\\
	(ii)  $\sigma_{max}(B) < \sigma_{min}(A),$
	 where $\sigma_{max}(.)$ and $\sigma_{min}(.)$ denotes the maximum and minimum singular values of any real entries square matrix, respectively; \\
	 (iii) $\rho (\vert A^{-1} \vert . \vert B \vert)<1$; \\
	 (iv) $\sigma_{max}(A^{-1}B) < 1,$ \\ where A is an invertible matrix.
\end{thm}

% % % % % % % % % % % % % % % % % % % % % % % % % % % % % % % % % % % % % % % % % % % % %

\begin{thm}\label{thm 2} \cite{tang2022solvability}
     The GAVME (\ref{Equ GAVME}) has exactly one solution if any of the following cases happen:\\
	(i) the column  $\mathcal{W}$-property is hold by $\{Q+P,-Q+P \};$ \\
	(ii)  $Q+P$ is nonsingular and the column  $\mathcal{W}$-property is hold by $\{I,(Q+P)^{-1}(-Q+P)\}$ with $n^{2}-$ order identity matrix I; \\
	(iii) $Q+P$ is nonsingular and matrix $(Q+P)^{-1}(-Q+P)$ is a P-matrix;\\
	(iv)  $\{(Q+P)F_1+(-Q+P)F_2\}$ is invertible, where $F_1,F_2$ $\in$ $\mathcal{R}^{n^{2} \times n^{2}}$ are two any non-negative diagonal matrices with $diag(F_1+F_2) > 0,$ with $Q=I \otimes B$ and $P=I \otimes A$.
\end{thm}

% % % % % % % % % % % % % % % % % % % % % % % % % % % % % % % % % % % % % % % % % % % % %

\begin{thm} \label{thm 4} \cite{tang2022solvability}
		In the case that each diagonal entry of $P+Q$ has the same sign as its corresponding entry of $P-Q$ then GAVME (\ref{Equ GAVME}) has exactly one solution if any of the following statements is true: \\
	(i)	$Q+P$ and $-Q+P$ are strictly diagonally dominant by columns;  \\
	(ii) $Q+P$, $-Q+P$ and all their column representative matrices are irreducibly diagonally dominant by columns, where  $Q=I \otimes B$ and $P=I \otimes A$.
\end{thm}

% % % % % % % % % % % % % % % % % % % % % % % % % % % % % % % % % % % % % % % % % % % % %

\begin{thm} \label{thm 5} \cite{tang2022solvability}
	If non-singular matrix A satisfies the condition
	$\rho((I \otimes A^{-1}B)D) < 1$ for any diagonal matrix $D=diag(d_{i}) \in \mathcal{R}^{n^{2} \times n^{2}}$ with $d_{i} \in [-1,1],$ then the GAVME (\ref{Equ GAVME}) has exactly one solution.
\end{thm}

% % % % % % % % % % % % % % % % % % % % % % % % % % % % % % % % % % % % % % % % % % % % %

\begin{thm} \label{thm 7} \cite{wang2021new}
The Sylvester-like AVE (\ref{Equ Sylvester-like AVE}) has exactly one solution if 
\begin{equation}
	\sigma_{max}(LK^{-1})\sigma_{max}(A^{-1}B)<1, 
\end{equation}
where A and K are non-singular matrices.
\end{thm}

% % % % % % % % % % % % % % % % % % % % % % % % % % % % % % % % % % % % % % % % % % % % 
% % % % % % % % % % % % % % % % % % % % % % % % % % % % % % % % % % % % % % % % % % % % 
% % % % % % % % % % % % % % % % % % % % % % % % % % % % % % % % % % % % % % % % % % % % 

\section{Main Results} 
%In this section, first, we provide the spectral radius condition for the unique solvability of the GAVME (\ref{Equ GAVME}). We then will obtain some necessary and sufficient unique solvability conditions for NGAVME (\ref{Equ NGAVME}). Finally, we correct a result for the Sylvester-like AVE (\ref{Equ Sylvester-like AVE}) provided in \cite{Wang and Li Sylvester 2021}.

The following Lemma is required to prove the spectral radius condition for the unique solvability of GAVME.

\begin{lem} \cite{wu2021unique} \label{lem 1}
	 If the following condition holds
\begin{equation*}
	 \rho (\vert A^{-1} B \vert)<1,
\end{equation*}
	 then GAVE (\ref{Equ GAVE}) has exactly one solution for arbitrary f, where matrix A is invertible.
\end{lem} 

 We have the following result based on Lemma \ref{lem 1}.

\begin{thm} \label{Main Thoerem}
	If invertible matrix A and  matrix B satisfy the condition
	  \begin{equation}\label{Equ5}
	\rho (\vert A^{-1} B \vert)<1,
	\end{equation}
then GAVME (\ref{Equ GAVME}) has exactly one solution for any matrix F.
\end{thm} 
\begin{proof}
 Let $X=(x_{1},...,x_{m})$ and $F=(f_{1},...,f_{m}),$ where $x_{j}$ and $f_{j}$	are the $j$th column of the matrices X and F, respectively. By $\vert X \vert=(\vert x_{1}\vert,...,\vert x_{m}\vert),$ GAVME (\ref{Equ GAVME}) can be rewritten as 
 $A(x_{1},...,x_{m}) + B(\vert x_{1}\vert,...,\vert x_{m}\vert)$ =  $(f_{1},...,f_{m}),$ or equivalently,
 \begin{equation} \label{Equ GAVME2}
	Ax_{j} + B \vert x_{j} \vert =f_{j}, 
\end{equation}
 where $j=1,2,...,m.$ When the condition $\rho (\vert A^{-1} B \vert)<1$ is true, then Lemma \ref{lem 1} allows   GAVE (\ref{Equ GAVME2}) to have a unique solution for each k. From this point, we may compute each $x_k$ individually.
\end{proof}

% % % % % % % % % % % % % % % % % % % % % % % % % % % % % % % % % % % % % % % % % % % % %
% % % % % % % % % % % % % % % % % % % % % % % % % % % % % % % % % % % % % % % % % % % % %
% % % % % % % % % % % % % % % % % % % % % % % % % % % % % % % % % % % % % % % % % % % % %

%%%%%%%%%%%%%%%%%%%%%%%%%%%%%%%%%%%%%%%%%%%%%%%%%%%%%%%%%%%%%%%%%%%%%%%%%%%%%%%%%%%%%%%%
%From Theorem \ref{Main Thoerem}, we can quickly get the following Corollary.
%\begin{cor} \label{Main Corollary}
%If A is an invertible matrix and satisfies the following condition
%	\begin{equation}
%	\rho (\vert A^{-1} \vert)<1,
%	\end{equation}
%	then absolute value matrix equation $AX + \vert X \vert =F$ has exactly one solution for any matrix F.
%\end{cor}
%%%%%%%%%%%%%%%%%%%%%%%%%%%%%%%%%%%%%%%%%%%%%%%%%%%%%%%%%%%%%%%%%%%%%%%%%%%%%%%%%%%%%%%%

In Theorem \ref{Main Thoerem}, we have a new result for the unique solvability of (\ref{Equ GAVME}). In practice, our condition is easy to verify, whereas conditions given in Theorem \ref{thm 2} and Theorem \ref{thm 5} have limited practical uses. In Theorem \ref{thm 2}, we have to determine the determinants of all column representative matrices of the given set and to apply Theorem \ref{thm 5}, we need to find all diagonal matrix D, which is not an easy task.

 Further, based on Theorem \ref{thm 01}, we have $\rho(\vert A^{-1}B \vert)$ $\leq$ $\rho(\vert A^{-1} \vert .\vert B \vert)$. This implies that condition $\rho(\vert A^{-1}B \vert)$ is slighter superior than $\rho(\vert A^{-1} \vert .\vert B \vert)$. To check our condition's validity, we are considering Example \ref{Ex2} and Example \ref{Ex1}. We are comparing our spectral radius condition (\ref{Equ5}) with all three conditions given in Theorem \ref{thm 1}. We observe that, our spectral radius condition (\ref{Equ5}) is slighter superior than the condition $\rho(\vert A^{-1} \vert . \vert B \vert) < 1$,  and in some instances, the performance of our condition is better than the singular value condition shown in Theorem \ref{thm 1}.

 % % % % % % % % % % % % % % % % % % % % % % % % % % % % % % % % % % % % % % % % % % % % %
 % % % % % % % % % % % % % % % % % % % % % % % % % % % % % % % % % % % % % % % % % % % % %
 % % % % % % % % % % % % % % % % % % % % % % % % % % % % % % % % % % % % % % % % % % % % %
 
% For some instances our condition perform better to the condition of Kumar et al. [Natl. Acad. Sci. Lett. (2023) 1-3].
 
 \begin{ex} \label{Ex2} \cite{tang2022solvability}
 	Let 
 	\begin{equation*}
 		A=
 		\begin{bmatrix}
 			5 &  -1\\
 			-4  &  4
 		\end{bmatrix}
 		~ ~  and ~ ~	B=
 		\begin{bmatrix}
 			-0.5 &  1 \\
 			0.5 & -2
 		\end{bmatrix}.
 	\end{equation*}	
 
Then sufficient condition $\rho(\vert A^{-1} B \vert)=0.38826 <1$ holds. So the GAVME (\ref{Equ GAVME}) has a unique solution. Further, $\rho(\vert A^{-1} \vert . \vert B \vert)= 1 \nless 1$, $\sigma_{max}(B)=2.3354 \nless 2.1939 = \sigma_{min}(A)$, and $\sigma_{max}(\vert B \vert)=2.3354 \nless 2.1939 = \sigma_{min}(A)$. The three conditions of the Theorem \ref{thm 1} are not holding, so these conditions fail to evaluate the unique solution of the GAVME (\ref{Equ GAVME}).
 \end{ex}

% % % % % % % % % % % % % % % % % % % % % % % % % % % % % % % % % % % % % % % % % % % % %
% % % % % % % % % % % % % % % % % % % % % % % % % % % % % % % % % % % % % % % % % % % % %
% % % % % % % % % % % % % % % % % % % % % % % % % % % % % % % % % % % % % % % % % % % % %
%	\begin{equation*}
%	\begin{bmatrix}
%		5 & -1  \\
%		-4 &  4 
%	\end{bmatrix}
%	\begin{bmatrix}
%		x_1 & x_2  \\
%		x_3 &  x_4 
%	\end{bmatrix}
%	+
%	\begin{bmatrix}
%		-0.5 & 1    \\
%		0.5 & -2 
%	\end{bmatrix}
%	\begin{bmatrix}
%		|x_1| & |x_2|  \\
%		|x_3| &  |x_4|
%	\end{bmatrix}
%	=
%	\begin{bmatrix}
%		0 & 0 \\
%		0 & 0 
%	\end{bmatrix}
%\end{equation*}
 
 % % % % % % % % % % % % % % % % % % % % % % % % % % % % % % % % % % % % % % % % % % % % %
 % % % % % % % % % % % % % % % % % % % % % % % % % % % % % % % % % % % % % % % % % % % % %
 % % % % % % % % % % % % % % % % % % % % % % % % % % % % % % % % % % % % % % % % % % % % %

\begin{ex} \label{Ex1}
	Consider the following GAVME (\ref{Equ GAVME})
	\begin{equation*}
		\begin{bmatrix}
			2 & -4 & 0 \\
			0 &  1.2 & 1.1 \\
			-2 &  0.8 & 0
		\end{bmatrix}
		\begin{bmatrix}
			x_1 & x_2 & x_3 \\
			x_4 &  x_5 & x_6 \\
			x_7 &  x_8 & x_9
		\end{bmatrix}
		+
		\begin{bmatrix}
			1 & -1 & 0   \\
			0 & 1 & 1 \\
			-1 & 0 & 0
		\end{bmatrix}
		\begin{bmatrix}
			|x_1| & |x_2| & |x_3| \\
			|x_4| &  |x_5| & |x_6| \\
			|x_7| &  |x_8| & |x_9|
		\end{bmatrix}
		=
		\begin{bmatrix}
		   -5.5 & 9  & 1 \\
			0.8 & 3.8 & 1.8 \\
			3.4 & -4.6 & -5.2
		\end{bmatrix}
	\end{equation*}	
	
	Clearly, $\rho(\vert A^{-1} B \vert)=0.9091 <1.$ Further,  $\sigma_{max}(A^{-1}B) = 1.0885 \nless 1,$  $\sigma_{max}(\vert B \vert)$= 1.8019 $\nless$ $\sigma_{min}(A)=0.9038$ and $\sigma_{max}(B)= 1.8019$ $\nless$ $\sigma_{min}(A)=0.9038.$ Here, conditions of Theorem (\ref{thm 1}) fail to evaluate  to judge the unique solution of the GAVME (\ref{Equ GAVME})
	This GAVME has a unique solution X = 	
	\begin{equation*}
		\begin{bmatrix}
			-3  & 1 & 2 \\
			0.5 & -2 & 1 \\
			-3 & 2 & -4
		\end{bmatrix}
		.
	\end{equation*}
\end{ex}

%%%%%%%%%%%%%%%%%%%%%%%%%%%%%%%%%%%%%%%%%%%%%%%%%%%%%%%%%%%%%%%%%%%%%%%%%%%%%%%%%%%%%%%%
\subsection{Conditions for the Uniqueness of the NGAVME Solution}
%Here, we discuss the unique solvability of the NGAVME (\ref{Equ NGAVME}). By the assumption that matrix C is invertible, then NGAVME (\ref{Equ NGAVME}) can be written as the following GAVE form

%Here, we discuss the unique solvability of the NGAVME (\ref{Equ NGAVME}). To obtain the unique solvability conditions for the NGAVME (\ref{Equ NGAVME}), we utilize the results of the GAVME (\ref{Equ GAVME}).

Now, we discuss the unique solvability of the NGAVME (\ref{Equ NGAVME}). To obtain the unique solvability conditions for the NGAVME (\ref{Equ NGAVME}), we utilize the equivalence between the NGAVME (\ref{Equ NGAVME}) and the GAVME (\ref{Equ GAVME}).

By assuming that matrix C is invertible, NGAVME (\ref{Equ NGAVME}) can be written as the following equivalent GAVME form
\begin{equation}\label{Equ NGAVME 1}
	AC^{-1}Y+B\vert Y \vert=F,~ with ~ Y=CX.
\end{equation}

%To obtain the result of unique solvability for NGAVME (\ref{Equ NGAVME}), we use the Equ. (\ref{Equ NGAVME 1}) into Theorem \ref{thm 1} to \ref{thm 5} and Theorem \ref{Main Thoerem}.

% % % % % % % % % % % % % % % % % % % % % % % % % % % % % % % % % % % % % % % % % % % % %
Based on Theorem \ref{thm 1}, we have the following result.
\begin{thm}\label{Thm 1} 
The NGAVME (\ref{Equ NGAVME}) has exactly one solution if any of the following conditions are satisfied:\\
	(i)  $\sigma_{max}(\vert B \vert) < \sigma_{min}(AC^{-1}),$ where C is invertible; \\
	(ii) $\sigma_{max}(B) < \sigma_{min}(AC^{-1})$, where C is invertible; \\
	(iii) $\rho (\vert CA^{-1} \vert . \vert B \vert)<1$;\\
	(iv) $\sigma_{max}(CA^{-1}B) < 1,$ \\ where A is an invertible matrix.
\end{thm}
\begin{proof}
 To prove the above result, we will use Equ. (\ref{Equ NGAVME 1}) into Theorem \ref{thm 1}. Since Equ. (\ref{Equ NGAVME 1}) is equivalent to GAVME form \ref{Equ GAVME}, so by Theorem \ref{thm 1}, our results of the Theorem \ref{Thm 1} holds.
	%The proof is direct hold by using Equ. (\ref{Equ NGAVME 1}) into Theorem \ref{thm 1}.		
\end{proof}

% % % % % % % % % % % % % % % % % % % % % % % % % % % % % % % % % % % % % % % % % % % % %
	
\begin{thm}\label{Thm 2} 
	The NGAVME (\ref{Equ NGAVME}) has exactly one solution if any of the following conditions are satisfied:\\
	(i)  $\{R+S,R-S\}$ holds the column  $\mathcal{W}$-property;\\
	(ii)  $R+S$ is invertible and  the column  $\mathcal{W}$-property is maintained by $\{I,(R+S)^{-1}(R-S)\}$ with $n^{2}-$ order identity matrix I; \\
	(iii) $R+S$ is invertible and matrix $(R+S)^{-1}(R-S)$ is a P-matrix;\\
	(iv)  $\{(R+S)F_1+(R-S)F_2\}$ is invertible, where $F_1,F_2$ $\in$ $\mathcal{R}^{n^{2} \times n^{2}}$ are two any non-negative diagonal matrices with $diag(F_1+F_2) > 0,$ where $R=I \otimes AC^{-1}$ and $S=I \otimes B.$
\end{thm}
\begin{proof}
	By using Equ. (\ref{Equ NGAVME 1}) into Theorem \ref{thm 2}, then results of Theorem \ref{Thm 2} automatically holds.		
\end{proof}
% % % % % % % % % % % % % % % % % % % % % % % % % % % % % % % % % % % % % % % % % % % % %

\begin{thm} \label{Thm 4} 
	In the case that each diagonal entry of $R+S$ has the same sign as its corresponding entry of $R-S$ then NGAVME (\ref{Equ NGAVME}) has exactly one solution if any of the following statements is true: \\
	(i)	$R+S$ and $R-S$ are strictly diagonally dominant by columns;  \\
	(ii) $R+S$, $R-S$ and all their column representative matrices are irreducibly diagonally dominant by columns, where $R=I \otimes AC^{-1}$ and $S=I \otimes B.$
\end{thm}
\begin{proof}
Since (\ref{Equ NGAVME 1}) is equivalent to the GAVME (\ref{Equ GAVME}), then Theorem \ref{thm 4} is applied to (\ref{Equ NGAVME 1}), and then we get our results.
%With the help of Theorem \ref{thm 4} and  Equ. (\ref{Equ NGAVME 1}), proof is completes.		
\end{proof}
% % % % % % % % % % % % % % % % % % % % % % % % % % % % % % % % % % % % % % % % % % % % %

\begin{thm} \label{Thm 5} 
	Let non-singular matrices A and C satisfy the condition
	$\rho((I \otimes CA^{-1}B)D) < 1$ for any diagonal matrix $D=diag(d_{i}) \in \mathcal{R}^{n^{2} \times n^{2}}$ with $d_{i} \in [-1,1],$ then the NGAVME (\ref{Equ NGAVME}) has exactly one solution.
\end{thm}
\begin{proof}
	The proof of the above result can be obtained directly by the Theorem \ref{thm 5} and  Equ. (\ref{Equ NGAVME 1}).		
\end{proof}
% % % % % % % % % % % % % % % % % % % % % % % % % % % % % % % % % % % % % % % % % % % % %

\begin{thm} \label{Thm 6}
	The NGAVME (\ref{Equ NGAVME}) has exactly one solution if the condition  $\rho(\vert CA^{-1}B \vert)<1$ holds, where A is an invertible matrix.
\end{thm}
\begin{proof}
	The proof is directly hold by using  the Theorem \ref{thm 7} and  Equ. (\ref{Equ NGAVME 1}).	
%	The proof is of the above result can be obtained by the Theorem \ref{thm 7} and  Equ. (\ref{Equ NGAVME 1}).		
\end{proof}
% % % % % % % % % % % % % % % % % % % % % % % % % % % % % % % % % % % % % % % % % % % % %
By the properties of singular values, Kronecker product and Theorem \ref{thm 001}, we have the following calculation: 
% % % % % % % % % % % % % % % % % % % % % % % % % % % % % % % % % % % % % % % % % % % % %
%\begin{equation*}
%\begin{split}	
%	\rho((I \otimes CA^{-1}B)D) & \leq \sigma_{max}((I \otimes CA^{-1}B)D) \\ & \leq  \sigma_{max}(I \otimes CA^{-1}B)\sigma_{max}(D) \\ & \leq  \sigma_{max}(I \otimes CA^{-1}B) \\ & =   \sigma_{max}(I) \sigma_{max}(CA^{-1}B) \\ & = \sigma_{max}(CA^{-1}B).
%\end{split}
%\end{equation*}

	$\rho((I \otimes CA^{-1}B)D) \leq$ 
	 $\sigma_{max}((I \otimes CA^{-1}B)D) \leq$
	  $\sigma_{max}(I \otimes CA^{-1}B)\sigma_{max}(D) \leq$
	   $\sigma_{max}(I \otimes CA^{-1}B) =$
	    $\sigma_{max}(I) \sigma_{max}(CA^{-1}B) =$
	     $\sigma_{max}(CA^{-1}B).$
% % % % % % % % % % % % % % % % % % % % % % % % % % % % % % % % % % % % % % % % % % % % %

%By above expression and Theorem \ref{Thm 6}, we get the following result.
We get the following result by the above expression and Theorem \ref{Thm 6}. 
\begin{prop} \label{Thm 06}
	The NGAVME (\ref{Equ NGAVME}) has exactly one solution if the condition  $\sigma_{max}(CA^{-1}B)<1$ holds, where A is an invertible matrix.
\end{prop}

Although the condition of Proposition \ref{Thm 06} is also given in Theorem \ref{Thm 1}, here we use a different approach.
By relation $\sigma_{max}(A) \sigma_{min}(A^{-1})=1$ for invertible matrix A, the following result is also can be used instead of the  Proposition \ref{Thm 06}.
\begin{cor} \label{Cor 1}
	The NGAVME (\ref{Equ NGAVME}) has exactly one solution if the condition  $\sigma_{min}(B^{-1}AC^{-1})>1$ holds, where B and C are invertible.
\end{cor}

% % % % % % % % % % % % % % % % % % % % % % % % % % % % % % % % % % % % % % % % % % % % %
% % % % % % % % % % % % % % % % % % % % % % % % % % % % % % % % % % % % % % % % % % % % %
%To validity of our results, we are considering the following small problem, although our result can be use for the larger problems in practice.

In practice, we may apply our results to the larger problems. To determine the validity of our results, we are considering the following small problem.

\begin{ex} \label{Ex3} 
	Consider the following matrices 
	\begin{equation*}
		A=
		\begin{bmatrix}
			-5 & 2 & 8\\
			1 & 2 & 3 \\
			7 & -5 & 0  
		\end{bmatrix},
		~B=
		\begin{bmatrix}
			1 & 2 & 0\\
			0 & 1 & -1 \\
			-1 & 2 & 0
		\end{bmatrix},
	~C=
	\begin{bmatrix}
		2 & 0 & 0\\
		0 & 1 & 0 \\
		0 & 0 & 2
	\end{bmatrix}
~~ and ~~	F=
\begin{bmatrix}
	14 & -7 & 19\\
	12 & 4 & 3 \\
	1 & 39 & -12
\end{bmatrix}.
	\end{equation*}	

By simple calculations, we have,
$\sigma_{max}(CA^{-1}B)= 0.90873 <1$, $\rho(\vert CA^{-1}B \vert) = 0.70285 <1$, and $\sigma_{min}(B^{-1}AC^{-1})= 1.1004 >1$. The conditions of Theorem \ref{Thm 6}, Theorem \ref{Thm 06} and Corollary \ref{Cor 1}, respectively, are satisfying here, so the corresponding NGAVME has a unique solution. The unique solution of the NGAVEM (\ref{Equ NGAVME}) is
	\begin{equation*}
	X=
	\begin{bmatrix}
		2 &  5 & -1 \\
		3 & -2 & 1  \\
		1 & 1 & 1
	\end{bmatrix}.
\end{equation*}
	
\end{ex}

% % % % % % % % % % % % % % % % % % % % % % % % % % % % % % % % % % % % % % % % % % % % %
% % % % % % % % % % % % % % % % % % % % % % % % % % % % % % % % % % % % % % % % % % % % %
\begin{rem}
 Clearly, Theorem \ref{Thm 1} to  Thoerem \ref{Thm 6}  are the generalization forms of the results in Theorem \ref{thm 1} to Theorem \ref{thm 5} including Theorem \ref{Main Thoerem} respectively. Moreover, the condition of the Theorem \ref{Thm 6} is slighter superior to the (iii) condition of Theorem \ref{Thm 1}, which is easily verified by the Theorem \ref{thm 01}.
\end{rem}

% % % % % % % % % % % % % % % % % % % % % % % % % % % % % % % % % % % % % % % % % % % % %
% % % % % % % % % % % % % % % % % % % % % % % % % % % % % % % % % % % % % % % % % % % % %

Now, our focus on the sufficient condition provided by Wang et al. \cite{wang2021new} for the unique solution of the Sylvester-like AVE (\ref{Equ Sylvester-like AVE}), the statement of Wang et al. \cite{wang2021new} given as:

\begin{thm} \label{thm 8} \cite{wang2021new}
	The Sylvester-like AVE (\ref{Equ Sylvester-like AVE}) has exactly one solution if 
	\begin{equation} \label{SLAVE2}
		\sigma_{min}(LK^{-1})\sigma_{min}(A^{-1}B)>1,
	\end{equation}
	where A and K are non-singular matrices.
\end{thm}

In general, the result of the Theorem \ref{thm 8} is incorrect. Let consider the example for n=1, the matrix $A=1$, $B=2$, $K=1$ and $L=1$ satisfy the condition (\ref{SLAVE2}) as $\sigma_{min}(LK^{-1})\sigma_{min}(A^{-1}B)=2>1,$ but Sylvester-like AVE $X + 2 \vert X \vert =1$ has a two solutions and $X + 2 \vert X \vert = -1$ has no solution.

Now, we are giving the corrected version of the Theorem \ref{thm 8}.

\begin{thm} \label{Thm 7}
	The Sylvester-like AVE (\ref{Equ Sylvester-like AVE}) has exactly one solution if 
	\begin{equation} \label{SLAVE3}
		\sigma_{min}(KL^{-1})\sigma_{min}(B^{-1}A)>1,
	\end{equation}
	where B and L are non-singular matrices.
\end{thm}

\begin{proof}
	Let $A_{1}= LK^{-1}$ and $A_{2}=A^{-1}B$, then by the simple relation $\sigma_{max}(B_1)\sigma_{min}(B_1^{-1})=1$ for non-singular matrix $B_1$, applying to Theorem \ref{thm 7}, we get our result.
\end{proof}

To the validation of Theorem \ref{Thm 7}, let consider the Example 1 and Example 2 from \cite{wang2021new}, by the simple calculation we get $\sigma_{min}(KL^{-1})\sigma_{min}(B^{-1}A)=1.3334>1,$ and $\sigma_{min}(KL^{-1})\sigma_{min}(B^{-1}A)=1.0153>1$ for Example 1 and Example 2, respectively. So our result is valid for examining the unique solution of Sylvester-like AVE (\ref{Equ Sylvester-like AVE}).
% % % % % % % % % % % % % % % % % % % % % % % % % % % % % % % % % % % % % % % % % % % % %
% % % % % % % % % % % % % % % % % % % % % % % % % % % % % % % % % % % % % % % % % % % % %
% % % % % % % % % % % % % % % % % % % % % % % % % % % % % % % % % % % % % % % % % % % % %
% % % % % % % % % % % % % % % % % % % % % % % % % % % % % % % % % % % % % % % % % % % % %

\section{Conclusions} \label{section 4}

This research presented a spectral radius condition for the unique solvability of the GAVME. By utilizing the results of GAVME, some results were obtained to ensure the unique solvability of the NGAVME. The published works in \cite{dehghan2020matrix,kumar2022note,kumar2023note,tang2022solvability,xie2021unique} are generalized here. \\
% The numerical results for the GAVME, NGAVME and Sylvester-like AVE are also exciting topics in the future.\\

% % % % % % % % % % % % % % % % % % % % % % % % % % % % % % % % % % % % % % % % % % % % %
% % % % % % % % % % % % % % % % % % % % % % % % % % % % % % % % % % % % % % % % % % % % %
% % % % % % % % % % % % % % % % % % % % % % % % % % % % % % % % % % % % % % % % % % % % % 
%{\bf Statements and Declarations.}

%{\bf Competing Interests.} The authors declare that there is not any conflict of interests regarding the publication of this manuscript.

%{\bf Data availability statements.}
%Data sharing not applicable to this article as no data sets were generated or analyzed during the current study.

%{\bf Acknowledgment.} This research work of the first author is supported by the Ministry of Education (Government of India) through GATE fellowship registration No. MA19S43033021.

% % % % % % % % % % % % % % % % % % % % % % % % % % % % % % % % % % % % % % % % % % % % %
% % % % % % % % % % % % % % % % % % % % % % % % % % % % % % % % % % % % % % % % % % % % %
% % % % % % % % % % % % % % % % % % % % % % % % % % % % % % % % % % % % % % % % % % % % %

% % % % % % % % % % % % % % % % % % % % % % % % % % % % % % % % % % % % % % % % % % % % %
% % % % % % % % % % % % % % % % % % % % % % % % % % % % % % % % % % % % % % % % % % % % %
% % % % % % % % % % % % % % % % % % % % % % % % % % % % % % % % % % % % % % % % % % % % %

%\nocite{*}
\label{Bibliography}

%\lhead{\emph{Bibliography}} % Change the page header to say "Bibliography"

\bibliographystyle{unsrt} % Use the "custom" BibTeX style for formatting the Bibliography

\bibliography{BibliographyAVME} % The references (bibliography) information are stored in the file named "Bibliography.bib"
%\lhead{\emph{Publications}}
%\include{Publications}
%\include{appendices}
\end{document}